\ifx\pdfoutput\undefined\else\pdfoutput=1\fi
\documentclass[9pt,a4paper,oneside,reqno]{amsart} 
\input{packages.tex}
\definecolor{col1}{HTML}{d7191c}
\definecolor{col2}{HTML}{fdae61}
\definecolor{col3}{HTML}{abdda4}
\definecolor{col4}{HTML}{2b83ba}

\DeclareMathOperator{\Tr}{Tr}

\DeclareMathOperator{\E}{\mathbf{E}}
\DeclareMathOperator{\Prob}{\mathbf{P}}

\newcommand{\ii}{\mathrm{i}}

\ifdefined\C
\renewcommand{\C}{\mathbf{C}}
\else
\newcommand{\C}{\mathbf{C}}
\fi

\newcommand{\R}{\mathbf{R}}
\newcommand{\N}{\mathbf{N}}
\newcommand{\Z}{\mathbf{Z}}

\newcommand{\DD}{\mathbf{D}}

\newcommand{\cO}{\mathcal{O}}
\newcommand{\co}{{\scriptstyle\mathcal{O}}}

\newcommand{\dif}{\operatorname{d}\!{}}
\DeclarePairedDelimiter{\braket}{\langle}{\rangle}%
\DeclarePairedDelimiter{\abs}{\lvert}{\rvert}%
\DeclarePairedDelimiter{\norm}{\lVert}{\rVert}%
\providecommand\given{}
\newcommand\SetSymbol[1][]{\nonscript\:#1\vert\allowbreak\nonscript\:\mathopen{}}
\DeclarePairedDelimiterX{\tuple}[1](){\renewcommand\given{\SetSymbol[\delimsize]}#1}
\DeclarePairedDelimiterX{\set}[1]{\{}{\}}{\renewcommand\given{\SetSymbol[\delimsize]}#1}
\DeclarePairedDelimiterX{\Set}[1]\{\}{\renewcommand\given{\SetSymbol[\delimsize]}#1}
\DeclarePairedDelimiterXPP{\landauO}[1]{\cO}(){}{#1}
\DeclarePairedDelimiterXPP{\landauo}[1]{\co}(){}{#1}
\DeclarePairedDelimiterXPP{\landauok}[1]{\co_k}(){}{#1}
\DeclarePairedDelimiterXPP{\landauOprec}[1]{\cO_\prec}(){}{#1}
\DeclarePairedDelimiterXPP{\Exp}[1]{\E}[]{}{\renewcommand\given{\SetSymbol[\delimsize]}#1}
\DeclarePairedDelimiterXPP{\condProb}[1]{\Prob}[]{}{\renewcommand\given{\SetSymbol[\delimsize]}#1}
\DeclareFontFamily{U}{mathx}{\hyphenchar\font45}
\DeclareFontShape{U}{mathx}{m}{n}{
      <5> <6> <7> <8> <9> <10>
      <10.95> <12> <14.4> <17.28> <20.74> <24.88>
      mathx10
      }{}
\DeclareSymbolFont{mathx}{U}{mathx}{m}{n}
\DeclareFontSubstitution{U}{mathx}{m}{n}
\DeclareMathAccent{\widecheck}{0}{mathx}{"71}
\usepackage[giveninits=true,url=false,doi=false,isbn=false,eprint=true,datamodel=mrnumber,sorting=nty,maxcitenames=4,maxbibnames=99,backref=false,block=space,backend=biber,bibstyle=phys]{biblatex} 
\AtEveryBibitem{\clearfield{month}}
\AtEveryCitekey{\clearfield{month}} 
\renewbibmacro{in:}{}
\DeclareFieldFormat[article]{title}{\emph{#1}} 
\DeclareFieldFormat{mrnumber}{\ifhyperref{\href{http://www.ams.org/mathscinet-getitem?mr=#1}{\nolinkurl{MR#1}}}{\nolinkurl{#1}}}
\DeclareFieldFormat{pmid}{\ifhyperref{\href{https://www.ncbi.nlm.nih.gov/pubmed/#1}{\nolinkurl{PMID#1}}}{\nolinkurl{#1}}}
\DeclareFieldFormat{eprint}{\ifhyperref{\href{https://arxiv.org/abs/#1}{\nolinkurl{arXiv:#1}}}{\nolinkurl{#1}}}
\renewbibmacro*{doi+eprint+url}{%
  \iftoggle{bbx:doi}{\printfield{doi}}{}
  \newunit\newblock%
  \printfield{mrnumber}%
  \newunit\newblock%
  \printfield{pmid}%
  \newunit\newblock%
  \printfield{eprint}%
  \iftoggle{bbx:url}{\usebibmacro{url+urldate}}{}}

\pgfplotsset{select coords between index/.style 2 args={
    x filter/.code={
        \ifnum\coordindex<#1\fi
        \ifnum\coordindex>#2\fi
    }
}}
\bibliography{references}

\newtheorem{theorem}{Theorem}
\newtheorem{lemma}[theorem]{Lemma}
\newtheorem{remark}[theorem]{Remark}
\newtheorem{corollary}[theorem]{Corollary}
\newtheorem{conjecture}[theorem]{Conjecture}
\numberwithin{theorem}{section}

\newcommand{\f}{\mathfrak{f}}
\newcommand{\g}{\mathfrak{g}}

\date{\today}
\author{Giorgio Cipolloni \and L\'aszl\'o Erd\H{o}s}
\address{IST Austria, Am Campus 1, 3400 Klosterneuburg, Austria}
\author{Dominik Schr\"oder\(^{\ast}\)}
\address{Institute for Theoretical Studies, ETH Zurich, Clausiusstr.\ 47, 8092 Zurich, Switzerland}
\email{giorgio.cipolloni@ist.ac.at}  
\email{lerdos@ist.ac.at} 
\email{dschroeder@ethz.ch}
\thanks{\(^\ast\)Supported by Dr.\ Max R\"ossler, the Walter Haefner Foundation and the ETH Z\"urich Foundation}
\subjclass[2020]{60B20, 68W40, 65F35} 
\keywords{Smoothed Analysis, Supersymmetric formalism, Circular Law}
\title[On the condition number of the  
shifted real Ginibre ensemble]{On the condition number  of the shifted real Ginibre ensemble}
\date{\today} 
\begin{document}  
\thispagestyle{empty}

\begin{abstract}  
    We derive an accurate lower tail estimate on the lowest singular value 
    $\sigma_1(X-z)$ of a \emph{real} Gaussian (Ginibre) random matrix $X$ shifted by a \emph{complex}
    parameter $z$. Such shift  effectively changes
    the upper tail behaviour of the condition number $\kappa(X-z)$ from the slower $\Prob(\kappa(X-z)\ge t)\lesssim 1/t$ decay 
    typical for real Ginibre matrices 
    to the faster $1/t^2$ decay seen for complex Ginibre matrices
    as long as $z$ is away from the real axis. 
    This sharpens and resolves a recent conjecture
    in~\cite{2005.08930} on the regularizing effect of the real Ginibre
    ensemble with a genuinely complex shift.
    As a consequence we obtain an improved upper bound on the eigenvalue condition 
    numbers (known also as the eigenvector overlaps) 
    for real Ginibre matrices. The main  technical tool is a rigorous supersymmetric 
    analysis from our earlier work~\cite{Cipolloni2020}.
\end{abstract}

\maketitle

\section{Introduction}
The condition number $\kappa(X)= \| X\|\|X^{-1}\|$ 
of large $N\times N$ random matrices $X$ has been a central object in numerical linear algebra
at least since the pioneering work of Goldstine and von-Neumann~\cite{MR41539}, and Demmel~\cite{MR929546}. Demmel showed that for a large class of \emph{complex} random matrices $X$
the probability that $\kappa(X)$ is larger than a threshold  $t\gg 1$ 
decays as $1/t^2$, while for \emph{real} matrices the decay rate is slower, of order $1/t$.
While the dependence on $N$ was not optimal in Demmel's work, for the specific Gaussian 
case much more precise results are available.
Gaussian random matrices have  frequently been used as a test case
since often explicit formulas are available for their spectral distribution.

The simplest non-Hermitian random matrix model is the \emph{real} or \emph{complex Ginibre ensemble}, 
consisting of matrices with
independent identically distributed (i.i.d) Gaussian matrix elements.  We fix the
customary normalization, $\E x_{ab}=0$, $\E |x_{ab}|^2= N^{-1}$ that  guarantees
that the density of eigenvalues of $X$ converges to the uniform measure on the complex unit
disk (known as the \emph{Circular law}) and that the spectral radius  of $X$
converges to 1 with very  high  probability (these results also hold for non-Gaussian matrix elements, see
e.g.~\cite{MR773436,MR1428519,MR2409368,MR866352,MR863545,MR3813992,2012.05602}), c.f.~\cref{fig circ law}.
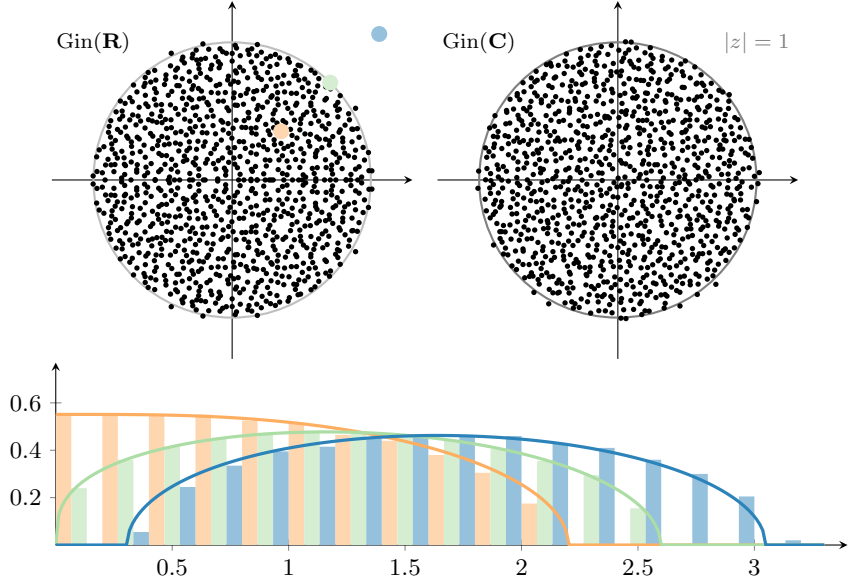
\begin{figure}
    \begin{tikzpicture}
    \begin{axis}[xmin=-1.3,xmax=1.3,ymin=-1.3,ymax=1.3,width=20em,height=20em,axis lines=middle,ticks=none]
        \addplot [only marks,draw=none,mark size=1pt,mark options={draw opacity=0,fill=black}] table [col sep=comma] {500r.csv};
        \addplot [only marks,draw=none,mark size=3pt,mark options={draw opacity=0,fill=col2!50}] [select coords between index={1}{1}] table [col sep=comma] {500zs.csv};
        \addplot [only marks,draw=none,mark size=3pt,mark options={draw opacity=0,fill=col3!50}] [select coords between index={2}{2}] table [col sep=comma] {500zs.csv};
        \addplot [only marks,draw=none,mark size=3pt,mark options={draw opacity=0,fill=col4!50}] [select coords between index={3}{3}] table [col sep=comma] {500zs.csv};
        \draw(axis cs:0,0) [draw=lightgray,thick] circle[radius=1];
        \node  at (axis cs:-1,1) {\small\(\mathrm{Gin}(\R)\)};
    \end{axis}
\end{tikzpicture}\quad
\begin{tikzpicture}
    \begin{axis}[xmin=-1.3,xmax=1.3,ymin=-1.3,ymax=1.3,width=20em,height=20em,axis lines=middle,ticks=none]
        \addplot [only marks,draw=none,mark size=1pt,mark options={draw opacity=0,fill=black}] table [col sep=comma] {500c.csv};
        \draw(axis cs:0,0) [draw=gray,thick] circle[radius=1];
        \node at (axis cs:1,1) {\small\color{gray}\(|z|=1\)};
        \node  at (axis cs:-1,1) {\small\(\mathrm{Gin}(\C)\)};
    \end{axis}
\end{tikzpicture}\\
\begin{tikzpicture} 
    \begin{axis}[width=12cm,height=4cm,axis lines=middle,xmax=3.4,ymax=.77%
        ]
        \addplot[ybar,bar width=.0666,fill=col2!50,bar shift=0.0333,area legend,draw=none] table[col sep=comma,x index=0,y index=2] {histSing.csv};
        \addplot[ybar,bar width=.0666,fill=col3!50,bar shift=0.1,area legend,draw=none] table[col sep=comma,x index=0,y index=3] {histSing.csv};
        \addplot[ybar,bar width=.0666,fill=col4!50,bar shift=0.1666,area legend,draw=none] table[col sep=comma,x index=0,y index=4] {histSing.csv};
        \addplot[col2,very thick,mark=none] table[col sep=comma,x index=0,y index=2] {rhosSing.csv};
        \addplot[col3,very thick,mark=none] table[col sep=comma,x index=0,y index=3] {rhosSing.csv};
        \addplot[col4,very thick,mark=none] table[col sep=comma,x index=0,y index=4] {rhosSing.csv};
    \end{axis}
\end{tikzpicture}
\caption{The top figure shows the eigenvalues of a single real and complex Ginibre matrix. Note that the eigenvalues of the real Ginibre matrix are symmetric with respect to the real axis, and that some (in fact \(\sim\sqrt{N}\)) eigenvalues are on the axis itself. The bottom figure shows the singular values of \(X-z\) for three different values of \(z\) in histogram form (for the single matrix whose eigenvalues are displayed in the top-left figure), together with their theoretic density (solid lines). The singular value density depends only on the absolute value \(\abs{z}\) and not on the phase of \(z\) since the effect of the real axis is only visible in the density of the smallest singular values, and not in the global density of all singular values. We note that  the singular value density is strictly positive at \(0\) whenever \(\abs{z}<1\).}\label{fig circ law}
\end{figure}
Edelman in~\cite{MR964668}
gave an exact formula  for the distribution of the lowest singular value of a Ginibre
matrix in both symmetry classes and derived precise large \(N\) asymptotics
for the condition number, confirming Demmel's upper tail decay on the distribution of  $\kappa(X)$
uniformly  in the dimension. 
Non-asymptotic upper and lower 
bounds with good explicit constants were obtained  in~\cite{MR2124157}
for the real case and later extended to rectangular 
matrices~\cite{MR2179688, MR2208323}
in both symmetry classes.

In more recent applications Ginibre matrices arise as  additive perturbations of a  deterministic matrix $A$.
The prominent example is the concept of \emph{smoothed analysis} (originally 
introduced  in~\cite{MR2145860} in the context of the simplex algorithm), where
Sankar, Spielman and Teng~\cite{MR2255338} considered 
the Gaussian elimination algorithm without pivoting for solving large dimensional linear systems of equations 
\(A{\bf x}={\bf b}\). 
The \emph{bit-complexity} of Gaussian elimination, i.e.\ the computational cost of achieving a desired output accuracy, depends primarily on $\kappa(A)$
and its upper tail is mainly  determined  by the lower tail behaviour of $\sigma_1(A)$, 
the lowest singular value of $A$ (note that $\sigma_1(A)^2 =\lambda_1(AA^*)$, the lowest eigenvalue of $AA^*$).
In order to  obtain a bound on the \emph{real world accuracy loss} of the problem  \(A{\bf x}={\bf b}\), 
rather than the \emph{averaged}
or \emph{worst case accuracy loss}, 
the main result of~\cite{MR2255338} is an estimate on the smoothed value of \(\kappa(A+\gamma X)\) 
for small \(\gamma\) and Ginibre-distributed \(X\).  In practice  $\gamma$ is then optimized to balance between the gain in bit-complexity versus the loss in precision. 

More recently smoothed analysis has been applied to the problem of
finding eigenvalue/eigenvector pairs~\cite{MR3801817} and to full matrix diagonalization~\cite{1906.11819,1912.08805} by Banks, Vargas, Kulkarni and Srivastava. This required to develop the ideas of  smoothed analysis for the \emph{eigenvector condition number} (see~\cref{kappa v} later)
in~\cite{1906.11819} and then in~\cite{1912.08805} further to the minimal eigenvalue gap (see~\cref{Delta def} later).

In~\cite{MR2255338}, 
the authors proved\footnote{The paper~\cite{MR2255338} states the result only for the real case, but the complex case easily follows by the same proof.}
the following lower tail bound on the (square of the) lowest singular value of the regularised matrix $A+\gamma X$:  
\begin{equation}\label{SSTreal1} 
    \Prob\biggl(\sigma_1(A+\gamma X)\le \frac{\sqrt{x}}{N}\biggr)\le C \sqrt{\frac{x}{\gamma^2}},
    \qquad x>0,
\end{equation}
for a real Ginibre matrix $X$, and
\begin{equation}\label{SSTcomplex1}
    \Prob\left(\sigma_1(A+\gamma X)\le \frac{\sqrt{x}}{N}\right)\le C \frac{x}{\gamma^2},
    \qquad x>0,
\end{equation}
for a complex Ginibre matrix $X$. 
The constant $C$ is universal, the estimates
are uniform in $A$ and $\gamma$.  The $N^{-2}$ scaling naturally comes from
the typical $1/N$ spacing between the eigenvalues of the corresponding Hermitized matrix
\begin{equation}
    H^A: = \begin{pmatrix}  0 &  A+\gamma X \cr  (A+\gamma X)^* & 0 \end{pmatrix}
\end{equation}
in its bulk spectrum.
Comparing 
the bounds~\cref{SSTreal1} and~\cref{SSTcomplex1} in the small $x$ regime,
note that the regularizing effect of a complex Ginibre matrix is much stronger.
Can one achieve the same  effect with real Ginibre matrices?

On one hand,  
inspecting the proof in~\cite{MR2255338},
the exponents of $x$ in the right hand side of~\cref{SSTreal1} and~\cref{SSTcomplex1} 
are direct consequences of the one- vs.\ two-dimensionality of the support of the real vs.\ complex
random variables $x_{ab}$ and the effect is completely independent of $A$.
On the other hand, quite remarkably, the local \emph{eigenvalue} statistics of the real and
complex Ginibre ensemble coincide away from the real axis, see~\cite[Theorem 11]{MR2530159}.
Very recently in~\cite{2105.13720, MR4235475} we showed an analogous phenomenon for the singular values of
the \emph{shifted Ginibre matrix}.  More precisely, in~\cite{2105.13720} the density of the low lying 
singular values of $X-z$  for a real and complex Ginibre $X$ was shown to coincide
if the shift parameter $z$ is genuinely complex, $|\Im z|\gg N^{-1/2}$. In the regime $|\Im z|\sim 1$ the same coincidence was proven for all $k$-point correlations functions~\cite[Theorem 2.8]{MR4235475}. In particular, on the level of
the small singular values,
the \emph{real} Ginibre matrix with a \emph{complex} shift behaves as a \emph{complex} Ginibre matrix!

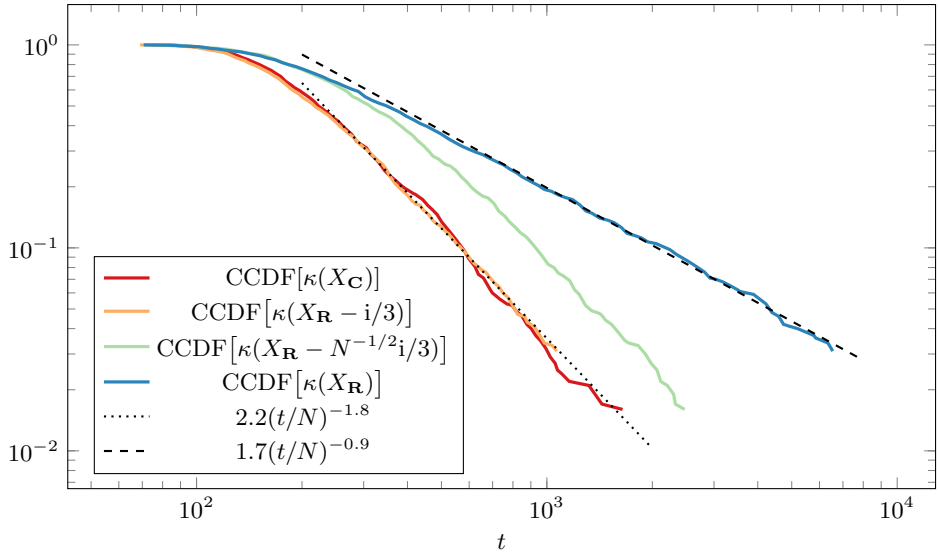
\begin{figure}[htbp]
    \centering
    \begin{tikzpicture}
        \begin{axis}[width=13cm,height=8cm,%
            xmode=log,ymode=log,legend pos=south west,xlabel=\(t\)
            ]
            \addplot[col1,very thick,mark=none] table[col sep=comma,x index=0,y index=1] {kappaDecay.csv};
            \addlegendentry{\(\mathrm{CCDF}[\kappa(X_\C)]\)}
            \addplot[col2,very thick,mark=none] table[col sep=comma,x index=2,y index=3] {kappaDecay.csv};
            \addlegendentry{\(\mathrm{CCDF}\bigl[\kappa(X_\R-\ii/3)\bigr]\)}
            \addplot[col3,very thick,mark=none] table[col sep=comma,x index=6,y index=7] {kappaDecay.csv};
            \addlegendentry{\(\mathrm{CCDF}\bigl[\kappa(X_\R-N^{-1/2}\ii/3)\bigr]\)}
            \addplot[col4,very thick,mark=none] table[col sep=comma,x index=4,y index=5] {kappaDecay.csv};
            \addlegendentry{\(\mathrm{CCDF}\bigl[\kappa(X_\R)\bigr]\)}
            \addplot [thick, domain=200:2000, samples=19,dotted] {2.26*(x/100)^(-1.8)};
            \addplot [thick, domain=200:8000, samples=19,dashed] {1.72*(x/100)^(-0.94)};
            \addlegendentry{\(2.2(t/N)^{-1.8}\)}
            \addlegendentry{\(1.7(t/N)^{-0.9}\)}
        \end{axis}
    \end{tikzpicture}
    \caption{The complementary cumulative distribution functions (CCDF) \(\mathbf P(\kappa\ge t)\) for the condition number \(\kappa(X-z)\) of shifted Ginibre matrices obtained from \(10\,000\) random matrices of size \(100\times 100\). Away from the real axis the probability of having condition number larger than \(t\) decays as \(t^{-2}\) for both real and complex Ginibre matrices. On the real axis the real Ginibre ensemble has a slower tail of \(t^{-1}\) and exhibits an interpolating behaviour as \(\Im z\sim N^{-1/2}\).}\label{kappa decay figure} 
\end{figure} 

For the purpose of the smoothed analysis this  indicates the possibility that real Ginibre matrices
are as effectively regularizing as the complex ones, at least  away from the real axis.
To test this hypothesis, we  consider the  simplest $A=-zI$ case, the shifted Ginibre ensemble. In fact, the following
conjecture in this spirit was very  recently posed in Section 7 of~\cite{2005.08930} (with our notations
and with $\gamma=1$ for simplicity):
\begin{equation}
    \label{conj}
    \Prob\left(\sigma_1(X-z)\le \frac{\sqrt{x}}{N}\right)\le C \frac{x}{ |\Im z|},
    \qquad x>0.
\end{equation}
Here and in the sequel we will frequently omit the identity matrix for brevity and write simply $X-z$ for $X-zI$. 
While~\cref{conj} highlights the role of $\Im z$, its scaling is far from optimal:
by analogy with the eigenvalues, one expects that a real Ginibre matrix $X$ near $z$ behaves 
essentially as a complex Ginibre matrix as long as $|\Im z|\gg N^{-1/2}$. Indeed, our main result
in~\cref{theo:bsmaleig} shows that
\begin{equation}
    \label{conj1}
    \Prob\left(\sigma_1(X-z)\le \frac{\sqrt{x}}{N}\right)
    \lesssim (1+\abs{\log x})x + e^{-\frac{1}{2}N|\Im z|^2}
    \min \Big\{\sqrt{x}, \frac{x}{\sqrt{N}|\Im z|} \Big\},
    \quad x>0,
\end{equation}
proving that the essentially linear bound (in $x$)
from~\cref{SSTcomplex1}  already dominates the tail behaviour
of the lowest singular value for $|\Im z|\gg N^{-1/2}$,  while the much larger $\sqrt{x}$ tail 
prevails in the opposite regime. 
Since $\kappa(X-z)  \sim \sigma_1(X-z)^{-1}$, we directly  obtain
the transition   from   $1/t$  to $1/t^2$ for
the upper tail  $\Prob ( \kappa(X-z) \ge t)$ as $|\Im z|$ increases well above $N^{-1/2}$, see~\cref{kappa decay figure}. 
A similar  behaviour is expected to hold for general matrix $A$, see~\cref{conj:smb}.

The bound~\cref{conj1} has several consequences on the \emph{eigenvalue condition number}
$\kappa(\lambda):= \| L\| \| R\|$ that determines the stability of the eigenvalue 
$\lambda$ against small perturbations,  where $L, R$ denote the corresponding left
and right  eigenvectors with  the customary normalization \(\braket{L,R}=1\).
For complex Ginibre matrices it is known that $\kappa(\lambda)$ is of order $\sqrt{N}$, see~\cite{PhysRevLett.81.3367,MR4095019,MR3851824}. For the real Ginibre case
a similar result is obtained in~\cite{MR3851824} but only for real eigenvalues $\lambda\in \R$.
Suboptimal bounds in $N$
have very recently been established in~\cite{2005.08930} and~\cite{2005.08908} that  also hold
for a general matrix $A$, in particular for any shift $z\in \C$.

The proof of our main estimate~\cref{conj1} uses the  supersymmetric (SUSY)
approach that is common in the physics literature on random matrices, but 
is less known  in the numerics community. Most of 
the necessary technical work has already been done in our previous paper~\cite{Cipolloni2020};
hence the current paper is short and focuses on the results. Our purpose  is to
demonstrate  the power of the SUSY  method 
to obtain very accurate estimates.
For example, the exponential suppression factor
$\exp(-\frac{1}{2} N|\Im z|^2)$ in~\cref{conj1} expressing the true effect of
the non-zero imaginary part of the shift parameter seems very hard to obtain 
with any other  method, while it easily comes out from the SUSY formalism.

\subsection*{Notations and conventions} 
For positive quantities \(f,g\) we write \(f\lesssim g\) and \(f\sim g\) if \(f \le C g\) or \(c g\le f\le Cg\), respectively, for some constants \(c,C>0\) which are independent of  $N$ and $z$.  We write \(\DD\subset \C\) for the open unit disk. We abbreviate the minimum and maximum of real numbers by \(a\wedge b:=\min\set{a,b}\) and \(a\vee b:=\max\set{a,b}\).  

\section{Main results}
We consider the ensemble $Y^z:=(X-z)(X-z)^*$ with $X\in  \R^{N\times N}$ being a \emph{real Ginibre} matrix, i.e.\ its entries $x_{ab}$ are such that $\sqrt{N}x_{ab}$ are i.i.d.\ standard real Gaussian random variables, and $z\in\C$ is a fixed complex parameter such that $|z|\le 1$. Our main results are 
an optimal lower tail estimate for the smallest singular value of \(X-z\) and its consequences for eigenvector overlaps and eigenvalue condition number of real Ginibre matrices. 

\subsection{Singular value and condition number tail estimates for \texorpdfstring{$X-z$}{X-z}}
The following theorem  gives an estimate on 
the lowest singular value of $X-z$ %
uniformly in all the relevant parameters and on the optimal scale. Its  direct
corollary is an analogous estimate on $\kappa(X-z)$.%

\begin{theorem}\label{theo:bsmaleig}
    Let $\eta:=\Im z$,  $\delta:=1-|z|^2$ and
    \begin{equation}\label{cn}
        c(N,\delta):=\frac{1}{N^{3/2}}\wedge \frac{1}{N^2|\delta|}.
    \end{equation}
    Then, denoting by $\lambda_1(Y^z)$ the smallest eigenvalue of $Y^z$, uniformly in $x\in [0, \infty)$, $\eta\in [-1,1]$, $\delta\in [-10N^{-1/2},1]$, it holds
    \begin{equation}\label{eq:optb}
        \Prob\left(\lambda_1(Y^z)\le x c(N,\delta)\right)\le C_*(1+\abs{\log x})x+C_*e^{-\frac{1}{2}N\eta^2} \Bigl(\sqrt{x} \wedge \frac{x}{\sqrt{N}\abs{\eta}}\Bigr)
    \end{equation}
    where $C_*$ is a universal constant.
\end{theorem}

Note that $c(N,\delta)$ is the correct scale of the typical size of $\lambda_1(Y^z)$. Indeed
the level spacing of the eigenvalues 
of $Y^z$ close to zero for $|z|<1$ is given by $N^{-2}\delta^{-1}$ and for $|z|=1$ by $N^{-3/2}$,
see~\cite[Section 5]{1907.13631}.
The  $N^{-3/2}$ scaling 
in the edge regime $|z|=1$ comes from the fact that
the density of eigenvalues of the  Hermitized matrix
\begin{equation}\label{eq:herm}
    H^z: = \begin{pmatrix}  0 &  X-z \cr  (X-z)^* & 0 \end{pmatrix}
\end{equation}
develops a cubic cusp singularity that has a natural  eigenvalue spacing $N^{-3/4}$.

\begin{remark} Introducing the coupling parameter $\gamma$ and thus  replacing
    $\lambda_1(Y^z)$ by $\lambda_1[(\gamma X-z)(\gamma X-z)^*]$ we then
    conclude a bound analogous to~\cref{eq:optb} after replacing $x$ by $x\gamma^{-2}$.
\end{remark}

\begin{remark}
    \cref{theo:bsmaleig} is proven only for matrices $X$ with Gaussian entries. However, the bound~\cref{eq:optb} can be  extended to matrices $X$ with generic independent identically distributed (i.i.d.) entries at the price of an additional error term. More precisely, for such matrices there exists $\omega>0$ such that  
    \begin{equation}
        \label{eq:aftgft}
        \Prob\big(\lambda_1(Y^z)\le x c(N,\delta)\big)\le \text{rhs.\ of~\cref{eq:optb}}+
        CN^{-\omega},
    \end{equation}
    for any $x\ge N^{-\omega}$. Given~\cref{eq:optb}, the bound in~\cref{eq:aftgft} is obtained by a standard Green function comparison (GFT) argument (see e.g.~\cite[Proposition 3]{MR4221653}).
\end{remark}

For a general deterministic matrix $A$ we can make
the following conjecture. Note that~\cref{eq:optb} proves~\cref{conj:smb} for the special case $A=-zI$ up to a logarithmic correction.

\begin{conjecture}\label{conj:smb}
    Let $X$ be an $N\times N$ real Ginibre matrix. There exist constants $c_*,C_*>0$ such that for any deterministic matrix $A$ and for any $\gamma>0$ it holds
    \begin{equation}
        \Prob\left(\sigma_1(\gamma X+A)\le \frac{\sqrt{x}}{N}\right) \le C_* \frac{x}{\gamma^2}+C_* e^{-c_*\Tr(\Im A)^2}\Bigl(\frac{\sqrt{x}}{\gamma}\wedge \frac{x}{\sqrt{\Tr(\Im A)^2}\gamma^2} \Bigr)
    \end{equation}
    where $\Im A := \frac{1}{2\ii}(A-A^\ast)$.
\end{conjecture}

We conclude this section by remarking that, by
$\kappa(X-z) = [\lambda_{\max }(Y^z)/\lambda_1(Y^z)]^{1/2}$,
from  a lower tail estimate~\cref{eq:optb} on $\lambda_1(Y^z)$ we immediately obtain
an upper tail bound on $\kappa(X-z)$ for \(\abs{z}<99/100\) (the complementary bound in the edge regime \(\abs{z}\approx 1\) follows similarly)
\begin{equation}\label{kappabound}
  \Prob\left(\kappa(X-z)\ge t \right)\lesssim \abs{\log t}\Bigl(\frac{N}{ t}\Bigr)^{2} 
    +e^{-\frac{1}{2}N\eta^2} \Bigl(\frac{N}{t} \wedge \frac{N^{3/2}}{\abs{\eta} t^2} \Bigr)+ e^{-N}.
\end{equation}
Here  we used that the largest eigenvalue $\lambda_{\max }(Y^z)$ can
be controlled by the large deviation bound  $\Prob (\| X\|\ge  K)\le e^{-\alpha K^2N}$
for some small $\alpha$ and any  large $K$. The bound~\cref{kappabound} shows 
Demmel's transition between the $1/t$ and $1/t^2$ tail behaviour
up to an exponentially small additive error.

\subsection{Overlaps and condition numbers for \texorpdfstring{\(X\)}{X}}
\begin{figure}[htbp]
    \centering
    \begin{tikzpicture}
        \begin{axis}[width=13cm,height=5cm, xtick={0,0.25,0.5,0.75,1}, ytick={0,.25,.5,.75,1.},xlabel=\(|z|\),axis lines=middle,xmax=1.3,
            ymin=0,ymax=1.2,%
            ]
            \addplot[col1, thick,error bars/.cd, y dir = both, y explicit,error bar style={very thick},error mark options={col1,rotate=90, thick} ] table[col sep=comma,x index=0,y index=1,y error index=2] {eOV.csv};
            \addplot[ thick,draw=col2,error bars/.cd, y dir = both, y explicit,error bar style={very thick},error mark options={col2,rotate=90, thick}] table[col sep=comma,x index=3,y index=4,y error index=5] {eOV.csv};
            \addplot [thick, domain=0:1, samples=19,gray,dashed] {1-x^2};
            \addlegendentry{\(\E[O_{ii}^{\C}/N]\)}
            \addlegendentry{\(\E\bigl[O_{ii}^{\R}/N\; \big\vert\; (\Im \lambda_i)^2\ge 10/N\bigr]\)}
            \addlegendentry{\(1-\abs{z}^2\)}
        \end{axis}
    \end{tikzpicture}
    \caption{The figure shows the empirical averaged overlap conditioned on \(|\lambda_i|=|z|\), as well as \((\Im \lambda_i)^2\ge 10/N\) in the real case, together 95\% confidence intervals, obtained from computing the eigenvalues for 10,000 Ginibre matrices of size \(100\times 100\). For the complex case, the
    corresponding behaviour \(\E \bigl[O_{ii}^\C/N\big\vert \lambda_i=z \bigr]\approx 1-|z|^2\) in the large $N$ limit has been established by Chalker and Mehlig~\cite{MR1755501,PhysRevLett.81.3367}. The figure above 
    suggests that the same relation holds true for real Ginibre matrices sufficiently far away from the real axis.}\label{exp Ov figure} 
\end{figure}
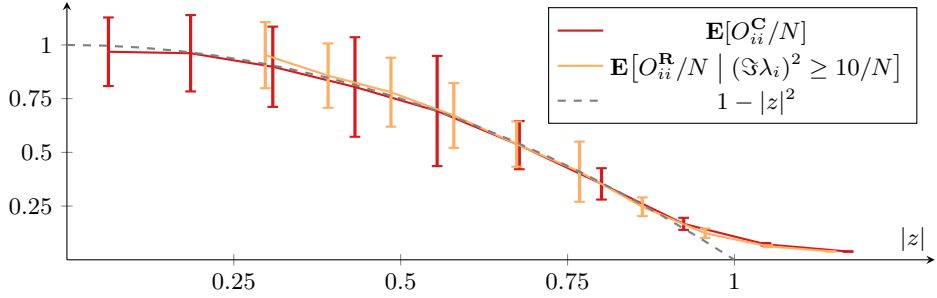

Our main result on the tail of the smallest singular value from~\cref{eq:optb} directly translates into optimal (up to logarithmic corrections) bounds on the \emph{eigenvector overlaps} and \emph{eigenvector condition number}. We denote the left- and right eigenvectors of \(X\) corresponding to an eigenvalue \(\lambda_i\) by \(L_i,R_i\) so that \(X=\sum \lambda_i L_i R_i^\ast\) with the normalization \(\braket{L_i,R_i}=1\) customary
in the theory of non-normal matrices. The \emph{diagonal eigenvector overlap} \(O_{ii}\) or \emph{eigenvalue condition number} \(\kappa(\lambda_i)\) are defined as
\begin{equation}
    O_{ii}:=\norm{L_i}^2\norm{R_i}^2=:\kappa(\lambda_i)^2.
\end{equation} 
The overlap \(O_{ii}\) is directly related to the stability of the eigenvalue \(\lambda_i=\lambda_i(X)\) 
under perturbations in the sense that 
\begin{equation}
    \lim_{\epsilon\to 0} \sup_{\norm{E}\le 1} \frac{\abs{\lambda_i(X+\epsilon E)-\lambda_i}}{\epsilon} = \sqrt{O_{ii}}. 
\end{equation}
\begin{theorem}\label{theo overlap}
    Let \(X\) be a real Ginibre matrix with left and right eigenvectors \(L_i,R_i\) corresponding 
    to the eigenvalue $\lambda_i$.
    Then      for any open set $\Omega\subset \DD$  and any $K>0$ we have 
    \begin{equation}\label{overlap upper bound}
        \Prob\Bigg( \sum_{i: \lambda_i\in\Omega} \norm{L_i}^2\norm{R_i}^2 \ge t(\log N)
        N^2 \int_{\Omega+ B(0, N^{-1/2}) }
        (1-\abs{z}^2)_+ \dif^2z \Bigg)\le CKt^{-1} + \frac{C_K}{N^{K}}
    \end{equation}
    for any $t>0$, 
    with some universal constant $C$ and $K$-dependent constants $C_K$.
    Here $\Omega+ B(0, N^{-1/2})$ denotes the Minkowski sum 
    of $\Omega$ with the ball of radius $N^{-1/2}$.
\end{theorem}
\begin{figure}[htbp]
    \centering
    \begin{tikzpicture}
        \begin{axis}[width=13cm,height=8cm,%
            xmode=log,ymode=log,xlabel=\(t\)
            ]
            \addplot[col2,very thick,mark=none] table[col sep=comma,x index=0,y index=1] {ovDecay.csv};
            \addplot[col3,very thick,mark=none] table[col sep=comma,x index=2,y index=3] {ovDecay.csv};
            \addplot[col4,very thick,mark=none] table[col sep=comma,x index=4,y index=5] {ovDecay.csv};
            \addlegendentry{\(\mathrm{CCDF}[O_{ii}^{\C}/N]\)}
            \addlegendentry{\(\mathrm{CCDF}\bigl[O_{ii}^{\R}/N\; \big\vert\; (\Im \lambda_i)^2\le 0.1/N\bigr]\)}
            \addlegendentry{\(\mathrm{CCDF}\bigl[O_{ii}^{\R}/N\; \big\vert\; (\Im \lambda_i)^2\ge 10/N\bigr]\)}
            \addplot [thick, domain=2:100, samples=19,dotted] {4.0*x^(-1.9)};
            \addplot [thick, domain=2:1000, samples=19,dashed] {1.9*x^(-0.9)};
            \addlegendentry{\(4.0t^{-1.9}\)}
            \addlegendentry{\(1.9t^{-0.9}\)} 
        \end{axis}
    \end{tikzpicture}
    \caption{The complementary cumulative distribution functions \(\mathbf P(O_{ii}/N\ge t)\) for the eigenvector overlaps of real and complex Ginibre matrices obtained from 10,000 matrices of size \(100\times 100\). The complex overlaps as well as the real overlaps away from the real axis share the same decay exponent of \(1.9\approx 2\), consistent with the fact~\cite{MR4095019} that in the complex case the overlaps are \(1/\gamma_2\)-distributed. The fatter tail of the real overlap close to the real axis is responsible for the fact~\cite{MR3851824} that \(\mathbf E[O_{ii}^\R\vert \lambda_i\in\R]=\infty\).}\label{ov decay figure} 
\end{figure}
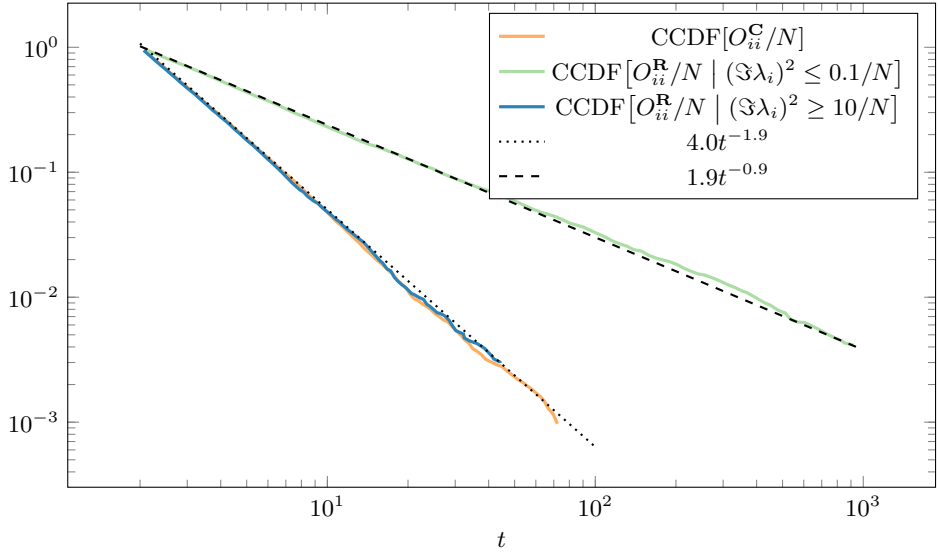 
\begin{remark}\leavevmode
    \begin{enumerate}[label=(\roman*)]
        \item For complex Ginibre matrices Chalker and Mehlig~\cite{MR1755501,PhysRevLett.81.3367} computed the expected overlap (rigorously only for \(z=0\)) and showed for its conditional expectation that 
        \begin{equation} 
            \Exp*{  O_{ii}^\C \given \lambda_i=z}=
            \Exp*{\norm{L_i}^2\norm{R_i}^2\given \lambda_i=z} \sim N(1-\abs{z}^2).
        \end{equation} 
        See~\cref{exp Ov figure} for a comparison between
        the      real Ginibre overlaps far away from the real 
        axis  and the
        complex Ginibre overlaps as a function of $|z|$.
        \item In the case of complex Ginibre matrices the distribution of individual overlaps \(\norm{L_i}^2\norm{R_i}^2\) has been identified in~\cite{MR4095019,MR3851824}, showing that 
        \[N^{-1}(1-\abs{\lambda_i}^2)^{-1}\norm{L_i}^2\norm{R_i}^2\] 
        converges in distribution to an inverse \(\gamma_2\) random variable. 
        \item Similarly, in the case of real Ginibre matrices the joint distribution of overlaps and their corresponding real eigenvalues has been identified in~\cite{MR3851824} via supersymmetric techniques. See~\cref{ov decay figure} for the tail decay of the $\mathrm{CCDF}$ of the overlaps corresponding to real and complex eigenvalues in the real Ginibre ensemble. Our numerics reproduce the $t^{-1}$ decay for the \(\mathrm{CCDF}\) of overlaps
        corresponding to real eigenvalues as shown in~\cite[Eq.~(2.2)]{MR3851824}.
        \item Recently, suboptimal versions of~\cref{overlap upper bound} could be established
        for the more general case of non-Hermitian random matrices $X$ with i.i.d.\ entries with deterministic 
        additive deformation \(X+A\) in~\cite{2005.08930} and~\cite{2005.08908},
        resulting in bounds of order \(N^5\) and \(N^3\), respectively. 
        \item Assuming~\cref{conj:smb} holds we expect that near-optimal versions of~\cref{overlap upper bound} can be established for general perturbations \(A+ X\) with real Ginibre matrices \(X\), potentially extending previous results~\cite{1906.11819} with complex Ginibre regularisation.
    \end{enumerate}
\end{remark}
Eigenvector overlaps can be used to estimate the \emph{eigenvector condition number}\footnote{The proof of the simple inequality can be found e.g.\ in~\cite[Lemma 3.1]{1906.11819}} 
\begin{equation}\label{kappa v}
    \kappa_v(X):=\inf_{VXV^{-1}=D} \norm{V}\norm{V^{-1}} \le \sqrt{N}\sqrt{\sum_{i=1}^N \norm{L_i}^2\norm{R_i}^2},
\end{equation}
where \(D\) is the diagonal matrix of eigenvalues. Thus we immediately obtain the following corollary.
\begin{corollary}\label{cor kappa v}
    For real Ginibre matrices \(X\) we have  
    \begin{equation}  
        \Prob\Big( \kappa_v(X) \ge t N^{3/2}\sqrt{\log N} \Big)\le CKt^{-2} +\frac{C_K}{N^K}, \qquad t>0,
    \end{equation}
    for any $K>0$    with some universal constant $C$ and $K$-dependent constants  $C_K$.
\end{corollary}
We note that it is generally expected that for dense random matrices \(\kappa_v(X)\) scales linearly in \(N\), c.f.~\cite[Page 338]{MR2155029}. Therefore our~\cref{cor kappa v} is still an overestimate by
a factor  \(N^{1/2}\sqrt{\log N}\), even though the estimate in~\cref{overlap upper bound} seems optimal. This is essentially due to the fact the ultimate inequality in~\cref{kappa v} loses a factor of \(\sqrt{N}\) by estimating operator norms by Frobenius norms. 

\subsection{Implications for numerical analysis}
Our results on the condition number have direct implications for the running time of various algorithms from numerical analysis. The conjugate gradient (CG) algorithm is commonly used for solving positive-definite linear systems~\cite{MR0060307}. Given an iid.\ random matrix \(X\) and a random vector \(\bm b\) the running time of the CG algorithm for solving the positive definite linear system \(X^\ast X\bm x=\bm b\) has empirically been shown to be universal~\cite{MR3276499} with respect to the distribution of the input data \(X\) (for rectangular matrices with real Gaussian and Bernoulli distributed entries). Theoretical bounds for the convergence rate of the CG algorithm for solving \(A\bm x=\bm b\), \(A>0\) in terms of the condition number are given by~\cite{MR3495481}
\begin{equation}
    \norm{\bm x_k-\bm x}\le 2\Bigl(\frac{\sqrt{\kappa(A)}-1}{\sqrt{\kappa(A)+1}}\Bigr)^{k}\norm{\bm x_0-\bm x},
\end{equation}
where \(\bm x_k\) is the \(k\)-th iterate of the CG algorithm. 
In~\cref{hist figure,CG tail fig} we study the running time of the CG algorithm for random matrices of the form \((X-z)^\ast(X-z)\) with real or complex shift \(z\), and \(X\) with real/complex Gaussian or Bernoulli entries. In agreement to our results on the condition number we find that real matrices with real shift lead to longer running times compared with real matrices with complex shift, or complex matrices with any shift.  
\begin{figure}[htbp]
    \centering  
    \begin{tikzpicture} 
        \begin{axis}[width=12cm,height=3cm,ymax=0.24, ytick={0,0.1,0.2},xmajorticks=false,axis lines=middle
            ]
            \addplot[ybar,bar width=.5,fill=black!10,bar shift=0.0,area legend] table[col sep=comma,x index=0,y index=1] {CG.csv};
            \addplot[ybar,bar width=.5,fill=black!30,bar shift=0.5,area legend] table[col sep=comma,x index=0,y index=2] {CG.csv};
            \addlegendentry{\(X_{\mathrm{Gin}(\R)}-1/2\)};
            \addlegendentry{\(X_{\mathrm{Gin}(\R)}-\ii/2\)};
        \end{axis}
    \end{tikzpicture}
    \begin{tikzpicture}
        \begin{axis}[width=12cm,height=3cm, ymax=0.24, ytick={0,0.1,0.2},xmajorticks=false,axis lines=middle
            ]  
            \addplot[ybar,bar width=.5,fill=black!10,bar shift=0.0,area legend] table[col sep=comma,x index=0,y index=3] {CG.csv};
            \addplot[ybar,bar width=.5,fill=black!30,bar shift=0.5,area legend] table[col sep=comma,x index=0,y index=4] {CG.csv};
            \addlegendentry{\(X_\mathrm{Ber}-1/2\)};
            \addlegendentry{\(X_\mathrm{Ber}-\ii/2\)};
        \end{axis}
    \end{tikzpicture}
    \begin{tikzpicture}
        \begin{axis}[width=12cm,height=3cm,ymax=0.24, ytick={0,0.1,0.2},xtick={140,160,180,200},axis lines=middle
            ]
            \addplot[ybar,bar width=1,fill=black!10,bar shift=0.0,area legend] table[col sep=comma,x index=0,y index=5] {CG.csv};
            \addlegendentry{\(X_{\mathrm{Gin}(\C)}-z,\;\abs{z}=1/2\)};
        \end{axis}
    \end{tikzpicture}
    \caption{The figure shows the distribution of the running time of the CG algorithm for solving the linear system \((X-z)(X-z)^\ast \bm x=\bm b\) for different choices of \(z\) and distributions of \(X\). The distributions have been obtained by sampling \(32\,000\) random matrices \(X\) of size \(100\times 100\) with independent (a) real Gaussian, (b) Bernoulli, and (c) complex Gaussian entries, and random vectors \(\bm b\) with iid uniformly distributed entries. The random matrices have been scaled such that their empirical spectrum is approximately uniformly distributed in the unit disc, and the random vectors \(\bm b\) have been normalised by their Euclidean norm. The horizontal axis shows the number of steps in the CG algorithm to reach a tolerance of \(10^{-8}\). While for complex random matrices \(X\) we observe no difference between real and complex shift, we find that for real random matrices both mean and fluctuation of the running time are smaller for the complex shift \(X-\ii/2\) than for the real shift \(X-1/2\). We note that the observed runtimes seem to be significantly influenced by rounding errors as the CG algorithm in exact arithmetic is guaranteed terminate in at most \(N=100\) steps~\cite{MR978581,MR1146656}.}\label{hist figure} 
\end{figure}
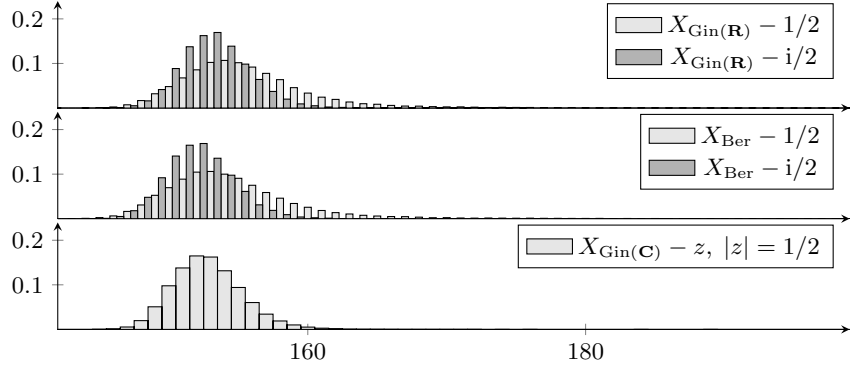
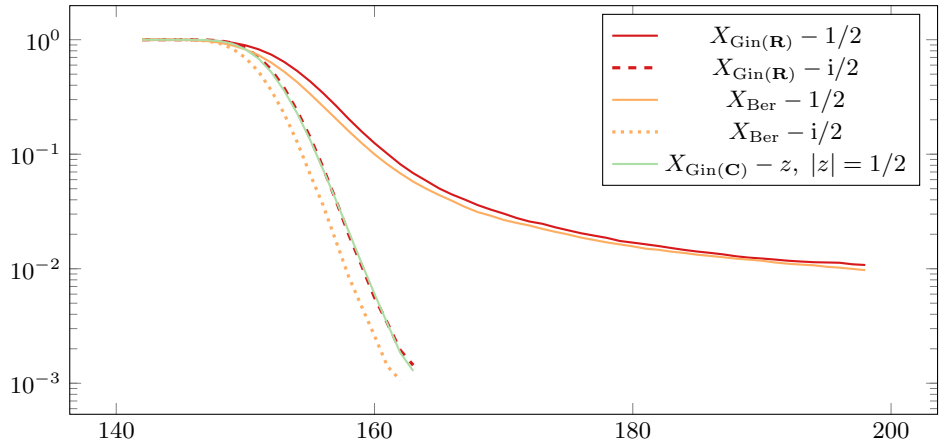
\begin{figure}[htbp]
    \centering
    \begin{tikzpicture}
        \begin{axis}[width=13cm,height=7cm,xtick={140,160,180,200},
            ymode=log
            ]
            \addplot[col1,thick,mark=none] table[col sep=comma,x index=0,y index=1] {CGtail.csv};
            \addplot[col1,very thick,mark=none,dashed] table[col sep=comma,x index=0,y index=2] {CGtail.csv};
            \addplot[col2,thick,mark=none] table[col sep=comma,x index=0,y index=3] {CGtail.csv};
            \addplot[col2,very thick,mark=none,dotted] table[col sep=comma,x index=0,y index=4] {CGtail.csv};
            \addplot[col3,thick,mark=none] table[col sep=comma,x index=0,y index=5] {CGtail.csv};
            \addlegendentry{\(X_{\mathrm{Gin}(\R)}-1/2\)};
            \addlegendentry{\(X_{\mathrm{Gin}(\R)}-\ii/2\)};
            \addlegendentry{\(X_\mathrm{Ber}-1/2\)};
            \addlegendentry{\(X_\mathrm{Ber}-\ii/2\)};
            \addlegendentry{\(X_{\mathrm{Gin}(\C)}-z,\;\abs{z}=1/2\)};
        \end{axis}
    \end{tikzpicture}
    \caption{The figure shows the tail behaviour of the CCDF of the running times from~\cref{hist figure}. It is evident that the running times for real matrices with real shift has a significantly heavier tail than those for real matrices with complex shift, or complex matrices with any shift.}\label{CG tail fig} 
\end{figure}

\section{Proof of~\cref{theo:bsmaleig}}\label{sec:1p1pf}

The proof of~\cref{theo:bsmaleig} relies on the following triple integral representation for
the expected trace of the resolvent $[Y^z+E]^{-1}= [Y^z+E\cdot I]^{-1}$
at any positive number $E>0$ (cf.~\cite[Eqs. (34)-(37)]{Cipolloni2020}):
\begin{equation}
    \label{realsusyexplAAr}
    \E \Tr   [Y^z+E]^{-1}=\frac{N}{4\pi \ii} \oint_\Gamma \dif \xi \int_0^{+\infty} \dif a\int_0^1 \dif \tau\frac{\xi^2a}{\tau^{1/2}} e^{N[f(\xi)-g(a,\tau,\eta)]} G_N(a,\tau,\xi,z),
\end{equation}
where $\Gamma$ any counter-clockwise contour around zero not crossing or encircling $-1$. Here
the functions $f$ and $g$ are given explicitly as
\begin{equation}
    \label{eq:fffr}
  f(\xi)=  f(\xi,E):=  E\xi+\log (1+\xi)-\log \xi-\frac{|z|^2}{1+\xi},      
\end{equation}
\begin{equation}
    \begin{split}
        \label{bosonphfAr}
      g(a,\tau,\eta)=  g(a,\tau,\eta,E)&:=  Ea+\frac{1}{2}\log [1+2a+a^2\tau]-\log a-\frac{1}{2}\log \tau\\
        &\quad -\frac{|z|^2(1+a)-2\eta^2 a^2 (1-\tau)}{1+2a+a^2\tau},
    \end{split}
\end{equation} 
where we denoted $\eta:=\Im z$. Furthermore, the function $G_N=G_N(a,\tau,\xi,\eta)$
is given by
\begin{equation}
    \label{eq:newbetG}
    \begin{split}
        G_N:= {}&\Bigl(
        N^2\frac{p_{2,0,0}}{a^2 \xi ^2 (\xi +1)^2 \tau }-N\frac{ p_{1,0,0}}{a^2 \xi ^2 (\xi +1) \tau  }+\delta  N^2\frac{p_{2,0,1}}{a \xi  (\xi +1)^2 \tau  }- N \delta \frac{p_{1,0,1}}{a \xi  (\xi +1) \tau  } \\
        &\quad + N^2\delta ^2\frac{ p_{2,0,2}}{(\xi +1)^2}  
        +  N^2\eta ^2 \frac{p_{2,2,0}}{a \xi  (\xi +1)^3 \tau }-N\eta ^2 \frac{p_{1,2,0}}{a \xi  \tau  }+  N^2\eta ^2\delta \frac{p_{2,2,1}}{(\xi +1) } 
        \Bigr) \\
        &\quad \times \Bigl((a^2\tau+2a+1)^2(\xi+1)^2\Bigr)^{-1},
    \end{split}
\end{equation}
where
\(\delta:= 1 -\abs{z}^2\) and 
 \(p_{i,j,k}=p_{i,j,k}(a, \tau, \xi)\) are explicit polynomials in \(a,\tau,\xi\) whose precise
form is not particularly relevant so we defer listing them to~\cref{appendix poly}.
 The indices $i,j,k$ in the definition of $p_{i,j,k}$ indicate the $N$, $\eta$ and $\delta$ powers, respectively.

The formula~\cref{realsusyexplAAr} looks somewhat complicated, but it is especially well suited 
for an accurate asymptotic analysis in the large $N$ regime via contour deformations and  Laplace asymptotics.
Note a remarkable reduction in the number of integration variables: while the expectation in the lhs.\ of~\cref{realsusyexplAAr}
involves $N^2$ real integrations, the rhs.\ is a three-fold integral.

We derived this formula  in~\cite[Section 3]{Cipolloni2020} using supersymmetric (SUSY) methods.
We will not repeat here the entire derivation of~\cref{realsusyexplAAr}, but we explain the main steps 
by  giving  a very short glimpse into SUSY.
The interested reader can find the detailed and self-contained proof of~\cref{realsusyexplAAr} in~\cite[Section 3.4]{Cipolloni2020}.

A fundamental identity on which SUSY methods rely is the following integral representation for the trace of the resolvent of 
any $N\times N$ Hermitian matrix $H$ for any spectral parameter\footnote{The imaginary part \(\Im w>0\) of the spectral parameter is a regularisation which can be removed later} $w\in\mathbf{C}$ with $\Im w> 0$:
\begin{equation}
\label{eq:firstintrep}
\mathrm{Tr}(H-w)^{-1}= \ii \int \braket{\chi,\chi} e^{-\ii\mathrm{Tr}[H-w](ss^*+\chi\chi^*)}, \qquad \int:=\int_{\mathbf{C}^N}\,\dif s\,\partial_\chi.
\end{equation}
We now explain the individual components in this formula.
 Here $s\in\mathbf{C}^N$ is a standard complex vector and $\int_{\mathbf{C}^N} \dif s$ is the usual $N$-fold complex area integral
 of the entries of $s$, e.g. $\dif s_1= \pi^{-1} \dif \Re s_1 \dif \Im s_1$.
 The letter
   $\chi$ denotes the  column vector with entries $\chi_1,\dots,\chi_N$, while 
   $\chi^*$ denotes the row vector with entries $\overline{\chi_1},\dots,\overline{\chi_N}$,
   where the collection of $2N$ symbols $\{ \chi_i, \overline{\chi_i} \,: \, i=1,2,\ldots, N\}$ consists of independent
 Grassmannian variables, i.e.\ they are $2N$ non-commuting algebraic  variables satisfying
\[
\chi_i\chi_j=-\chi_j\chi_i, \qquad \chi_i\overline{\chi_j}=-\overline{\chi_j}\chi_i,\qquad \overline{\chi_i}\overline{\chi_j}=-\overline{\chi_j}\overline{\chi_i}
\]
(the overline does not indicate any complex conjugation). 
These variables  naturally generate a $2^{2N}$-dimensional algebra over the complex scalar field.
For Grassmannian vectors $\chi,\phi$ we define the "scalar product"
\[
\braket{\chi,\phi}:=\sum_{i=1}^N\overline{\chi_i}\phi_i.
\]
Beyond polynomials, one may define analytic functions of 
Grassmannian variables via  power series, but notice that any polynomial of degree higher than $2N$ vanishes
since $\chi_i^2=0$, hence in practice any analytic function is a polynomial. For example 
\[\exp{(2\chi_1 + 3\chi_2\chi_3)}=
1 + 2\chi_1 + 3\chi_2 \chi_3+ \frac{1}{2} (2\chi_1 + 3\chi_2\chi_3)^2 = 1 + 2\chi_1 + 3\chi_2\chi_3 + 6\chi_1\chi_2\chi_3.
\]
One may also define a concept of integration (in the  sense of Berezin~\cite{MR914369}) over Grassmannian variables, which 
may be expressed equivalently  via introducing the (formal) derivatives
\[
\partial_{\chi_i}\chi_i=\partial_{\overline{\chi_i}}\overline{\chi_i}=1, \qquad \partial_{\chi_i}1=\partial_{\overline{\chi_i}}1=0, \qquad \partial_\chi:=\partial_{\chi_1}\partial_{\overline{\chi_1}}\dots \partial_{\chi_N}\partial_{\overline{\chi_N}}
\]
and extend them by multilinearity to all finite combinations of monomials in Grassmannians.

The integral representation in~\cref{eq:firstintrep} 
consists of $N$ complex integrals and $2N$ Grassmann integrals.
To heavily reduce the number of integrals, and obtain~\cref{realsusyexplAAr}, we rely on another
key ingredients of SUSY calculus, the
\emph{superbosonization formula}~\cite[Eq. (1.13)]{MR2430637}. 
We will not write it up here in full generality, but only in the form we need.

We now choose $H=Y^z$ in~\cref{eq:firstintrep} and define the $N\times 4$ matrix $\Psi:=(s,\overline{s},\chi,\overline{\chi})$. 
The main advantage of the r.h.s.\ of~\cref{eq:firstintrep}  over its l.h.s.\ is that Gaussian expectation
of a quadratic function (as $\Tr Y^z$)  in the exponent can be directly computed.
Taking the expectation of~\cref{eq:firstintrep}, and performing the Gaussian integration for the entries of $Y^z$, we are left with an integrand that turns out to be a meromorphic function depending on the variables
 only via  $\Psi^*\Psi$, which is the  $4\times 4$ matrix consisting of all the scalar products 
 of the four vectors $s, \overline{s},\chi,\overline{\chi}$.  This special form of the resulting function allows us
 to  use the superbosonization formula:
\begin{equation}
\label{eq:superbo}
\begin{split}
\int F(\Psi^*\Psi)={}&\int_Q \mathrm{SDet}^{N/2}(Q)F(Q), \qquad Q:=\left( \begin{matrix} x & \sigma \\ \tau & y\end{matrix}\right) \\
\int_Q:={}&\frac{1}{(2\pi)^2\ii}\int \dif x\oint\dif y\partial_\sigma \left(\frac{\mathrm{det}(y)}{\mathrm{det}(x)}\right)^{1/2} \mathrm{det}\left(1-\frac{x^{-1}}{y}\sigma\tau\right)^{1/2},
\end{split}
\end{equation}
for any meromorphic $F$ (see~\cite[Appendix A]{Cipolloni2020} for the extension of the superbosonization formula to meromorphic functions). Here $\mathrm{SDet}$ denotes the \emph{superdeterminant}, which is defined by
\[
\mathrm{SDet}\left( \begin{matrix} x & \sigma \\ \tau & y\end{matrix}\right):=\frac{\mathrm{det}(x)}{\mathrm{det}(y-\tau x^{-1}\sigma)}.
\]
The matrix $Q$ in~\cref{eq:superbo} is a $4\times 4$ matrix written 
as a  $2\times 2$ block matrix of $2\times 2$ blocks:  $x$ is non-negative Hermitian with $x_{11}=x_{22}$, $y$ is a scalar multiple of the identity matrix.
The off-diagonal block $\sigma$ is a $2\times 2$ matrix with fresh Grassmannian entries and $\tau$  is given by 
\[
\tau:=-\left( \begin{matrix} 0 & -1 \\ 1 & 0\end{matrix}\right)\sigma^t\left( \begin{matrix} 0 & 1 \\ 1 & 0\end{matrix}\right).
\]
In~\cref{eq:superbo}  $\int \dif x$ denotes the integral over the Lebesgue measure on non-negative Hermitian matrices with
the additional constraint  $x_{11}=x_{22}$, $\oint\dif y$ denotes the contour integral over $|y|=1$ in a counter-clockwise direction
(as before, we identified the scalar multiple of the identity matrix $y$ with the corresponding scalar)
and $\partial_\sigma:=\partial_{\sigma_{11}}\partial_{\sigma_{22}}\partial_{\sigma_{21}}\partial_{\sigma_{12}}$
denote the Grassmann derivatives.

Notice that the superbosonisation formula~\cref{eq:superbo} entails a drastic reduction in the number
of integration variables; the l.h.s.\ involves $N$ complex integration,  the r.h.s.\ has only four one-dimensional integrals
after all Grassmannians are eliminated.
In our concrete application one of these four integrals can be performed trivially, yielding eventually the three-fold integration
in~\cref{realsusyexplAAr}. For more details on SUSY calculus and for the complete proof of~\cref{realsusyexplAAr},
see the proof of~\cite[Eqs. (34)-(37)]{Cipolloni2020}.

Having explained our key formula~\cref{realsusyexplAAr}, we now conclude this section with the proof of~\cref{theo:bsmaleig}.

\begin{proof}[Proof of~\cref{theo:bsmaleig}]

The proof of this theorem is completely analogous to the proof of~\cite[Theorem 2.3]{Cipolloni2020} after replacing the bound in~\cite[Lemma 6.4]{Cipolloni2020} with the improved bound in~\cref{lem:remsmall} below. Furthermore, two technical estimates from~\cite[Lemma 5.2, Lemma 6.3]{Cipolloni2020} are used throughout the proof. These are straightforward but tedious bounds on
explicit integrals of the form
\[
\int_{\Gamma\setminus\widetilde{\Gamma}} \frac{e^{Nf(\xi)}}{\xi^b}\, \dif \xi \qquad \mathrm{and} \qquad \int_{N^\rho}^\infty\int_{N^{\rho/2}a^{-1}}^1 \frac{e^{-Ng(a,\tau,\eta)}}{a^{\alpha-1}\tau^{\gamma-1/2}}\, \dif \tau\dif a,
\]
respectively. Here $\Gamma$ is any contour around zero not crossing or encircling $-1$, 
\[
\widetilde{\Gamma}:=\{\xi\in\Gamma: |\xi|\le N^\rho\},
\]
with some small fixed constant $\rho>0$
and $b\ge 0$, $\alpha\ge 2$, $1\le \gamma\le \alpha$ are additional parameters. We will use these  technical lemmas also in the proof of~\cref{theo:bsmaleig}, but we will not repeat them here, just refer to~\cite{Cipolloni2020}. 
To make the presentation cleaner we only present the proof in the more critical case $\delta\in [0,1]$, the case $\delta\in [-10N^{-1/2},0)$ is analogous and is omitted (see e.g.~\cite[Section 6.2]{Cipolloni2020}). 

First we show that the regime where at least one among $a$, $\tau$ and $\xi$ is small gives
a negligible contribution to the triple integral. The proof of this lemma is postponed to~\cref{sec:addl}.
\begin{lemma}
\label{lem:remsmall}
Fix a small $\rho>0$. Let $\delta\in [0,1]$, $I=I_a:=[0,N^{\rho/2}a^{-1}]$, and let $c(N,\delta)$ be defined as in~\cref{cn}. Then, for any positive $E\lesssim c(N,\delta)$ it holds that
\begin{equation}
\label{eq:slimpb}
\begin{split}
&\int_\Gamma \dif \xi \int_0^{+\infty} \dif a\int_0^1 \dif \tau \frac{\xi^2a}{\tau^{1/2}} e^{N[f(\xi)-g(a,\tau,\eta)]} G_N(a,\tau,\xi,z)\\
&\quad=
\int_{\Gamma\setminus\widetilde{\Gamma}} \dif \xi \int_{N^\rho}^{+\infty} \dif a\int_{[0,1]\setminus I} \dif \tau \frac{\xi^2a}{\tau^{1/2}} e^{N[f(\xi)-g(a,\tau,\eta)]} G_N(a,\tau,\xi,z)+\landauO{\mathcal{E}}\\
&\mathcal{E}:= N^{5/2+\rho}e^{-\frac{1}{2}N^{1-2\rho}}\left(\frac{c(N,\delta)^{-1/2}e^{-(N\eta^2)/2}}{E^{1/2}\vee [c(N,\delta)^{1/2}N^{1/2}|\eta|]}+\frac{1+\abs{\log(NE^{2/3})}}{c(N,\delta)}\right).
\end{split}
\end{equation}
\end{lemma}

Using~\cref{lem:remsmall}, by~\cref{realsusyexplAAr}, we conclude that
\begin{equation}
\label{eq:maincontr}
\begin{split}
    &\E \Tr   [Y^z+E]^{-1}\\
    &=\frac{N}{4\pi \ii} \int_{\Gamma\setminus\widetilde{\Gamma}} \dif \xi \int_{N^\rho}^{+\infty} \dif a\int_{[0,1]\setminus I} \dif \tau\frac{\xi^2a}{\tau^{1/2}} e^{N[f(\xi)-g(a,\tau,\eta)]} G_N(a,\tau,\xi,z)+\mathcal{O}(N\mathcal{E}),
\end{split}
\end{equation}
with $\mathcal{E}$ defined in~\cref{eq:slimpb}. In particular, we are left only with the regime where all $a$, $\tau$, $\xi$ are large (in absolute value). In this regime we can use Taylor and Laurent  expansions for the functions
 $f$, $g$, $G_N$; these are listed in~\cite[Eqs. (75)-(77)]{Cipolloni2020}.
  Then we use the bounds from~\cite[Lemma 5.2]{Cipolloni2020} to estimate the regime $|\xi|\ge N^\rho$ and the ones from~\cite[Lemma 6.3]{Cipolloni2020} to estimate the regime $(a,\tau)\in [N^\rho,+\infty)\times ([0,1]\setminus I)$, which together imply
\begin{equation}
\label{eq:intermb}
\begin{split}
&\left|\frac{N}{4\pi \ii} \int_{\Gamma\setminus\widetilde{\Gamma}} \dif \xi \int_{N^\rho}^{+\infty} \dif a\int_{[0,1]\setminus I} \dif \tau\frac{\xi^2a}{\tau^{1/2}} e^{N[f(\xi)-g(a,\tau,\eta)]} G_N(a,\tau,\xi,z)\right|\\
&\qquad\quad\lesssim \frac{e^{-\frac{1}{2}N\eta^2}c(N,\delta)^{-1/2}}{\sqrt{E}\vee [c(N,\delta)^{1/2}N^{1/2}|\eta|]}+c(N,\delta)^{-1}\bigl[1+\abs{\log(NE^{2/3})}\bigr].
\end{split}
\end{equation}

Finally, combining~\cref{eq:maincontr} and~\cref{eq:intermb}, we readily conclude
  \begin{equation}
        \label{eq:befmar}
        \abs{\E \Tr [Y^z+E]^{-1} } \lesssim  \frac{e^{-\frac{1}{2}N\eta^2}c(N,\delta)^{-1/2}}{\sqrt{E}\vee [c(N,\delta)^{1/2}N^{1/2}|\eta|]}+c(N,\delta)^{-1}\bigl[1+\abs{\log(NE^{2/3})}\bigr],
    \end{equation}
    where we used the $N\mathcal{E}$ is always smaller than the r.h.s.\ of~\cref{eq:befmar}. Given~\cref{eq:befmar}, choosing $E=x c(N,\delta)$, the bound in~\cref{eq:optb} follows by a simple Markov inequality.
\end{proof}

\normalcolor

\section{Proof of~\cref{theo overlap}}
We begin with an a priori bound on the minimal eigenvalue gap 
\begin{equation}\label{Delta def}
    \Delta:=\min_{i\ne j}\abs{\lambda_i-\lambda_j}
\end{equation}
which follows directly from estimating the smallest two singular values of \(X-z\), see e.g.~\cite[Theorem 3.1.1]{ge} or~\cite[Theorem 1.9]{MR4235467}. 
\begin{lemma}\label{lemma Omega Delta}
    There exist constants \(c,C>0\) such that for Ginibre matrices \(X\) and for any $K\ge 100$ we have 
    \begin{equation}
        \Omega_\Delta:=\set{\Delta> N^{-3K}},\qquad \Prob(\Omega_\Delta^c) \le \frac{C}{N^{2K-5}}+Ce^{-cN}.
    \end{equation} 
\end{lemma}
Together with the singular value bound from~\cref{theo:bsmaleig} we obtain the following a priori bound on overlaps. 
\begin{lemma}\label{lemma Omega O}
    There exists a constant \(C>0\) such that for Ginibre matrices \(X\) the event
    \(\Omega_O:=\set{\max_{i}O_{ii}<N^{12K}}\) satisfies
    \begin{equation}
        \Prob(\Omega_O^c\cap \Omega_\Delta) \le \frac{C}{N^{K}},
    \end{equation}
for any $K\ge 100$.
\end{lemma}
\begin{proof}
    We claim that
    \begin{equation}\label{Omega O inclusion}
        \Omega_{O}^c\cap \Omega_\Delta \subset\bigcup_{z}^{N^{8K}} \set{\sigma_1(X-z)\le 2N^{-10K}},
    \end{equation}
    where the union is taken over \(z\)'s on an \(N^{-4K}\)-grid
    inside the unit disk. Indeed, for \(i:=\arg\max_j O_{jj}\) there exists \(z\in B(\lambda_i,N^{-4K})\) on the grid
    and we now show that $\sigma_1(X-z)\le 2N^{-10K}$ for this $z$. We
    use $\sigma_1(X-z)= \| (X-z)^{-1}\|^{-1}$  and  the spectral decomposition
    \[
    \frac{1}{X-z} = \sum_i \frac{R_i L_i^*}{\lambda_i-z}.
    \]
    Since on the event \(\Omega_\Delta\) for each \(j\ne i\) we have \(\abs{\lambda_j-z}\ge \abs{\lambda_i-\lambda_j}-\abs{\lambda_i-z}\ge N^{-3K}-N^{-4K}\ge N^{-3K}/2\) it follows that 
    \begin{equation}
        \begin{split}
            \frac{1}{\sigma_1(X-z)}&=\norm{(X-z)^{-1}} \ge \sqrt{O_{ii}}\Bigl(\frac{1}{\abs{\lambda_i-z}}-\sum_{j\ne i}\frac{1}{\abs{\lambda_j-z}}\Bigr) \\
            &\ge \sqrt{O_{ii}}(N^{4K}-2N^{3K+1})\ge \sqrt{O_{ii}}N^{4K}/2\ge N^{10K}/2,
        \end{split}
    \end{equation}
    confirming~\cref{Omega O inclusion}. Thus it follows from~\cref{theo:bsmaleig} and a union bound that 
    \begin{equation}
        \Prob(\Omega_{O}^c\cap \Omega_\Delta)\le C N^{8K-10K+2} \le \frac{C}{N^{K}},
    \end{equation}        
    concluding the proof of the Lemma. 
\end{proof}
Using the a priori bounds from~\crefrange{lemma Omega Delta}{lemma Omega O} we are ready to present the proof of~\cref{theo overlap}. The basic idea is to relate the eigenvalue overlaps to the area of the pseudo-spectrum of \(X\), see e.g.~\cite[Section 52]{MR2155029},~\cite[Section 3.6]{MR4095019},~\cite[Lemma 3.2]{1906.11819} or~\cite[Lemma 2.3]{2005.08908} for a quantitative version.
\begin{proof}[Proof of~\cref{theo overlap}]
    We introduce the event \(\Omega_{\Delta,O}:=\Omega_\Delta\cap\Omega_O\) and claim that on \(\Omega_{\Delta,O}\) for \(\epsilon:=N^{-12 K}\) we have 
    \begin{equation}\label{ov ineq}
        \sum_{\lambda_i\in\Omega} O_{ii} \le 4\frac{\abs{\set*{z\in \Omega+B(0,N^{-6K})\; : \; \sigma_1(X-z)\le \epsilon}}}{\epsilon^2}.
    \end{equation}
    Indeed, first note that
    \(B(\lambda_i,O_{ii}^{1/2}\epsilon/2)\cap B(\lambda_j,O_{jj}^{1/2}\epsilon/2)=\emptyset\) for \(i\ne j\) due to 
    \[\abs{\lambda_i-\lambda_j}\ge \Delta > N^{-3K}  > \epsilon(O_{ii}^{1/2}/2+O_{jj}^{1/2}/2). \]
    Then, for \(z\in B(\lambda_i,O_{ii}^{1/2}\epsilon/2 )\) we have 
    \[ 
       \frac{1}{\sigma_{1}(X-z)}=\norm{(X-z)^{-1}} \ge \frac{O_{ii}^{1/2}}{\abs{\lambda_i-z}}- \sum_{j\ne i} \frac{O_{jj}^{1/2}}{\abs{\lambda_j-z}}\ge \frac{2}{\epsilon}-N\frac{N^{6K}}{\Delta-\epsilon O_{ii}^{1/2}/2} \ge \frac{1}{\epsilon}  
    \]
    and from this relation it follows that
    \[\bigcup_{\lambda_i\in\Omega}B(\lambda_i,O_{ii}^{1/2}\epsilon/2)\subset \set*{z\in \Omega+B(0,N^{-6K})\; : \;
    \sigma_1(X-z)\le \epsilon}.\]
    Comparing the volumes of both sides we obtain~\cref{ov ineq}. 
    
    Now  from~\cref{ov ineq} and~\cref{eq:optb} we get 
    \begin{equation}
        \label{eq:finbhopb}
        \begin{split}
            &\Exp*{\sum_{\lambda_i\in \Omega}O_{ii}\given \Omega_{\Delta,O}}\\
            &\quad\lesssim \int_{\Omega+B(0,N^{-6K})} \frac{\condProb*{\sigma_1(X-z)\le \epsilon\given \Omega_{\Delta,O}}}{\epsilon^2}\dif^2z\\
            &\quad\lesssim \int_{\Omega+B(0,N^{-1/2})} \Bigg[(N^2 \big| 1-|z|^2\big| \vee N^{3/2}) K\log N \\
            &\qquad\quad +e^{-\frac{1}{2}N(\Im z)^2}
            \left(\frac{N\big| 1-|z|^2\big|^{1/2}\vee N^{3/4}}{\epsilon}\wedge \frac{N^{3/2}(1-|z|^2)_+
            \vee N}{|\Im z|}\right)\Bigg]\dif^2z \\
            &\quad\lesssim K\log N \left(N^2\int_{\Omega+B(0,N^{-1/2})} \big| 1-|z|^2\big|\dif^2 z \vee N^{3/2}|\Omega+B(0,N^{-1/2})|\right).
        \end{split}
    \end{equation}
    In the last inequality, in order to estimate the terms multiplied by $e^{-N(\Im z)^2/2}$, we performed the $\dif \Re z$ and $\dif \Im z$ integrations
    separately and split the analysis into three regimes: (i) $|1-|z|^2|\le N^{-1/2}$, (ii) $1-|z|^2> N^{-1/2}$ and $|\Im z|\ge N^{-1/2+\xi}$, for some small $\xi>0$, (iii) $1-|z|^2> N^{-1/2}$ and $|\Im z|\le N^{-1/2+\xi}$. The regime (i) is trivial since the factor $|1-|z|^2|$ can be neglected; for (ii) we used that  $e^{-N(\Im z)^2/2}\le e^{-N^\xi/2}$ and so 
    the contribution of this regime is exponentially small; finally in (iii) we used that $\big| 1-|z|^2\big|\le \big| 1- |\Re z|^2\big| $ and that
    \[
    \int_{\Pi(\widetilde{\Omega})} (1- y^2) \dif y \lesssim \sqrt{N} \int_{\widetilde{\Omega}} (1-|z|^2) \dif^2 z,
    \]
    where $\widetilde{\Omega}:=(\Omega+B(0,N^{-1/2}))\cap \{|\Im z|\le N^{-1/2+\xi}\}\cap \{1-|z|^2>N^{-1/2}\}$ and $\Pi(\widetilde{\Omega})\in\R$ is the projection of $\widetilde{\Omega}$ onto the real axis.
    Finally, by the estimate~\cref{eq:finbhopb} used on the event $\Omega_{\Delta,O}$ 
    together with a simple Markov inequality combined with the
    probability bound on the complement 
    $\mathbf{P}(\Omega_{\Delta,O}^c)\le N^{-K}$ from~\crefrange{lemma Omega Delta}{lemma Omega O}, we conclude~\cref{overlap upper bound}.
\end{proof}

\appendix

\section{Proof of~\cref{lem:remsmall}}
\label{sec:addl}

The proof of this lemma is very similar to the proof of~\cite[Lemma 6.4]{Cipolloni2020}, however we present a detailed proof here since we need the slightly improved bound~\cref{eq:slimpb} compared to~\cite[Eq. (92)]{Cipolloni2020}. More precisely, in the current paper we exploit the regularising effect of $\eta$ in the very small $E$ regime (see~\crefrange{eq:secondsmallb}{eq:thirdsmallb} below); notice the additional regularisation $c(N,\delta)^{1/2}N^{1/2}|\eta|$ in the denominator of the first error term in~\cref{eq:slimpb}.
The analogous term in~\cite[Lemma 6.4]{Cipolloni2020} had a $E^{-1/2}$ singularity for small $E$.

\begin{proof}[Proof of~\cref{lem:remsmall}]
Throughout the proof we choose $\Gamma$ as in~\cite[Eq. (48a)]{Cipolloni2020} to make the comparison clearer, i.e. $\Gamma=\Gamma_{z_*}:=\Gamma_{1,z_*}\cup\Gamma_{2,z_*}$ with
\[
\Gamma_{1,z_*}:=\set*{-\frac{2}{3}+\ii t :\given 0\le |t|\le \sqrt{|z_*|^2-\frac{4}{9}}}, \qquad \Gamma_{2,z_*}:=\set*{|z_*|e^{\ii \psi}:\psi\in [-\psi_{z_*},\psi_{z_*}]},
\]
where $\psi_{z_*}:=\arccos[2/(3|z_*|)]$, and $z_*=z_*(E,\delta)$ is defined in~\cite[Eq. (45)]{Cipolloni2020}. We remark that $|z_*|\sim E^{-1/3}\vee \sqrt{\delta E^{-1}}$, and recall that $\widetilde{\Gamma}:=\{\xi\in\Gamma: |\xi|\le N^\rho\}$.

Note that $g(a,\tau,\eta)\ge g(a,\tau,0)$ for any $a,\tau$ and that the map $\tau\mapsto g(a,\tau,0)$ is decreasing for any $a\in [0,+\infty)$ (see~\cite[(iii) of Lemma 6.1]{Cipolloni2020}). Hence, using that  $g(a,\tau,\eta)\ge g(a,1,0)=f(a)$, together with~\cite[(ii)-(iv) of Lemma 6.1]{Cipolloni2020}, we conclude that
\begin{equation}
\label{eq:expbneed}
\sup_{\xi\in \widetilde{\Gamma}}\left|e^{Nf(\xi)}\right|+\sup_{a\in [0,N^\rho]}\left|e^{-Ng(a,\tau,\eta)}\right|\lesssim e^{-Nf(N^\rho)}\lesssim e^{-\frac{1}{2}N^{1-2\rho}},
\end{equation}
where in the last inequality we also used that $f(N^\rho)=(\delta N^{-\rho}+(2N^{2\rho})^{-1})(1+\mathcal{O}(N^{-\rho}))$.

The proof of~\cref{eq:slimpb} is divided into three regimes: 
\begin{enumerate}[label=(\arabic*)]
\item \label{it:1} $a\in [0,N^\rho]$, $\xi\in \Gamma$, $\tau\in [0,1]$ (see~\cref{eq:fisrtsmallb} below),
\item \label{it:2} $a\ge N^\rho$, $\xi\in \Gamma$, $\tau\in I$ (see~\cref{eq:secondsmallb} below),
\item \label{it:3} $a\ge N^\rho$, $\xi\in \widetilde{\Gamma}$, $\tau\in [0,1]\setminus I$ (see~\cref{eq:thirdsmallb} below).
\end{enumerate}
We will now prove that all these three regimes give an exponentially small contribution.

To estimate the regime~\cref{it:1} we consider two further sub-cases: (i) $(a,\xi)\in [0,N^\rho]\times \widetilde{\Gamma}$, (ii) $(a,\xi)\in [0,N^\rho]\times \Gamma\setminus \widetilde{\Gamma}$. In case (i), we write $Ng(a,\tau,\eta)=(N-2)g(a,\tau,\eta)+2g(a,\tau,\eta)$. Then we use~\cref{eq:expbneed} to estimate $e^{-(N-2)g(a,\tau,\eta)}$ and $e^{Nf(\xi)}$ and we use $e^{-2g(a,\tau,\eta)}\lesssim a^2\tau$ for the remaining part. In this way, together with the bound $|G_N|\lesssim N^{2+5\rho}(a^2|\xi|^2\tau)^{-1}$ which follows by the explicit expression of $G_N$ in~\cref{eq:newbetG}, we readily get the estimate $N^{2+7\rho}e^{-N^{1-2\rho}/2}$. Here we used that the area of the domain of integration is bounded by $N^{2\rho}$. In case (ii), we first notice that by the explicit form of $G_N$ in~\cref{eq:newbetG} and the fact that $|\xi|\ge N^\rho$ we have
\[
\left|\int_{\Gamma\setminus \widetilde{\Gamma}} e^{Nf(\xi)} \xi^2 G_N(a,\tau,\xi,z)\, \dif \xi\right|\lesssim N^{2+5\rho}\frac{N^{1/2}\vee (N\delta)}{a^2\tau},
\]
by~\cite[Lemma 5.2]{Cipolloni2020}. Then, proceeding exactly as in case (i) to estimate the $(a,\tau)$-integral, we conclude that this regime is bounded by $N^{2+5\rho}(N^{1/2}\vee (N\delta))e^{-N^{1-2\rho}/2}$. Combining the bounds in cases (i) and (ii) we conclude that
\begin{equation}
\label{eq:fisrtsmallb}
\begin{split}
    &\abs*{\int_\Gamma \dif \xi \int_0^{N^\rho} \dif a\int_0^1 \dif \tau \frac{\xi^2a}{\tau^{1/2}} e^{N[f(\xi)-g(a,\tau,\eta)]} G_N(a,\tau,\xi,z)}\\
    &\qquad\lesssim N^{2+7\rho}(N^{1/2}\vee (N\delta))e^{-\frac{1}{2}N^{1-2\rho}}.
\end{split}
\end{equation}

Next we consider the regime~\cref{it:2}. In this case we will use that $g(a,\tau,0)\ge g(a, N^{\rho/2}a^{-1},0)$, by~\cite[(iii) of Lemma 6.1]{Cipolloni2020} and that by explicit computations
\begin{equation}
\label{eq:needexpb}
g(a, N^{\rho/2}a^{-1},0)=Ea+\frac{\delta}{N^{\rho/2}}+\frac{1}{2N^\rho}+\mathcal{O}\left(\frac{\delta}{N^\rho}+\frac{1}{N^{3\rho/2}}\right).
\end{equation}
Then, using~\cref{eq:expbneed} to estimate $e^{Nf(\xi)}$ when $\xi\in \widetilde{\Gamma}$, and~\cite[Lemma 5.2]{Cipolloni2020} for the regime $\xi\in \Gamma\setminus \widetilde{\Gamma}$, we get that
\begin{equation}
\left|\int_{\Gamma\setminus \widetilde{\Gamma}} e^{Nf(\xi)} \xi^2 G_N(a,\tau,\xi,z)\, \dif \xi\right|\lesssim C(N,\delta,a,\tau),
\end{equation}
with
\[
C(a,\tau)=C(N,\delta,a,\tau):=N^{2+5\rho}(N^{1/2}\vee (N\delta))\left(1+\frac{1}{a^2\tau}\right).
\]
We now write
\[
g(a,\tau,\eta)=g(a,\tau,0)+\frac{2\eta^2a^2(1-\tau)}{1+2a+a^2\tau}\ge g(a, N^{\rho/2}a^{-1},0)+\frac{2\eta^2a^2(1-\tau)}{1+2a+a^2\tau},
\]
and so, using that $e^{-2g(a,\tau,\eta)}\lesssim a\tau$ and that
\[
\frac{2\eta^2a^2(1-\tau)}{1+2a+a^2\tau}\ge \frac{\eta^2a}{2},
\]
by~\cref{eq:needexpb} we get
\begin{equation}
\label{eq:needtoint}
\begin{split}
&\abs*{\frac{a}{\tau^{1/2}}e^{-Ng(a,\tau,\eta)}\int_{\Gamma\setminus \widetilde{\Gamma}} e^{Nf(\xi)} \xi^2 G_N(a,\tau,\xi,z)\, \dif \xi}\\
&\quad\lesssim C(a,\tau)a^2\tau^{1/2}e^{-(N-2)g(a,\tau,\eta)}\\
&\quad\lesssim C(a,\tau)a^2\tau^{1/2} e^{-\frac{1}{2}N^{1-\rho}}e^{-(N-2)[Ea+\eta^2a/2]}.
\end{split}
\end{equation}
Computing the $(a,\tau)$-integral of the r.h.s.\ of~\cref{eq:needtoint}, we conclude that
\begin{equation}
\begin{split}
\label{eq:secondsmallb}
&\left|\int_\Gamma \dif \xi \int_{N^\rho}^\infty \dif a\int_I \dif \tau \frac{\xi^2a}{\tau^{1/2}} e^{N[f(\xi)-g(a,\tau,\eta)]} G_N(a,\tau,\xi,z)\right| \\
&\qquad\quad\lesssim N^{3/2+7\rho}\big(N^{1/2}\vee (N\delta)\big)e^{-\frac{1}{2}N^{1-2\rho}}\frac{e^{-(N\eta^2)/2}}{E^{1/2}\vee [c(N,\delta)^{1/2}N^{1/2}|\eta|]}.
\end{split}
\end{equation}
Notice that in~\cref{eq:secondsmallb} we estimated the integral of r.h.s.\ of~\cref{eq:needtoint} more precisely than in~\cite{Cipolloni2020}. This additional improvement, which is relevant only for small $E$, comes from using the regularising effect of $e^{-(N-2)\eta^2a/2}$.

Finally, in order to conclude the proof of this lemma, we are left with the regime~\cref{it:3}. In this case, since $|\xi|\le N^\rho$, we use~\cref{eq:expbneed} to bound $e^{Nf(\xi)}$ and~\cite[Lemma 6.3]{Cipolloni2020} to estimate the $(a,\tau)$-integral, with again exploiting the regularising effect of $\eta$. Hence, we conclude that
\begin{equation}
\label{eq:thirdsmallb}
\begin{split}
&\left|\int_{\widetilde{\Gamma}} \dif \xi \int_{N^\rho}^\infty \dif a\int_{[0,1]\setminus I} \dif \tau \frac{\xi^2a}{\tau^{1/2}} e^{N[f(\xi)-g(a,\tau,\eta)]} G_N(a,\tau,\xi,z)\right| \\
&\qquad\quad\lesssim N^{5/2+\rho}\big(N^{1/2}\vee (N\delta)\big)e^{-\frac{1}{2}N^{1-2\rho}}\\
&\qquad\qquad\quad\times \left(\frac{e^{-(N\eta^2)/2}}{E^{1/2}\vee [c(N,\delta)^{1/2}N^{1/2}|\eta|]}+\big[N^{1/2}\vee (N\delta)\big]\cdot \bigl[1+\abs{\log(NE^{2/3})}\bigr]\right).
\end{split}
\end{equation}

Combining~\cref{eq:fisrtsmallb},~\cref{eq:secondsmallb},~\cref{eq:thirdsmallb} we conclude~\cref{eq:slimpb}.
\end{proof}

\section{Explicit formulas for the real symmetric integral representation}\label{appendix poly}
Here we collect the explicit formulas for the polynomials of \(a,\xi,\tau\) in the definition of \(G_N\) in~\cref{realsusyexplAAr}.
\begin{align*}
    p_{2,0,0}&:=  a^4 \tau ^2+2 a^3 \xi  \tau +4 a^3 \tau -a^2 \xi ^2 \tau +4 a^2 \xi ^2+8 a^2 \xi +2 a^2 \tau \\
    &\qquad +4 a^2+2 a \xi ^3+8 a \xi ^2+10 a \xi +4 a+\xi ^4+4 \xi ^3+6 \xi ^2+4 \xi +1,   \\
    p_{1,0,0}&:=  - a^4 \xi \tau ^2+a^4 \tau ^2-2 a^3 \xi ^2 \tau -2 a^3 \xi  \tau +4 a^3 \tau -a^2 \xi ^3 \tau -3 a^2 \xi ^2 \tau \\
    &\qquad -2 a^2 \xi  \tau +4 a^2 \xi +2 a^2 \tau +4 a^2+2 a \xi ^2+6 a \xi +4 a+\xi ^3+3 \xi ^2+3 \xi +1,\\
    p_{2,2,0}&:=  4 (a+1) \left(a^2 \tau +a \xi  \tau +2 a \tau +\xi ^2+2 \xi +1\right), \\
    p_{1,2,0}&:=  4 (a+1)  \left(a^2 \tau +a \xi  \tau +2 a \tau +\xi +1\right), \\
    p_{2,0,1}&:=  2  \bigl(a^3 \tau ^2+2 a^2 \xi  \tau +4 a^2 \tau +2 a \xi ^2+2 a \xi  \tau  \\
    &\qquad\qquad+4 a \xi +3 a \tau +2 a+\xi ^3+4 \xi ^2+5 \xi +2\bigr)\\
    p_{1,0,1}&:=  2 \bigl(a^3 \tau ^2+2 a^2 \xi  \tau +4 a^2 \tau +a \xi ^2 \tau +3 a \xi  \tau \\
    &\qquad\qquad+2 a \xi +3 a \tau  +2 a+\xi ^2+3 \xi +2\bigr), \\
    p_{2,2,1}&:=  4 (a+1)  (a+\xi +2), \\
    p_{2,0,2}&:=  a^2 \tau +2 a \xi +4 a+\xi ^2+4 \xi +4.  
\end{align*}

\section*{Acknowledgement}
We would like to thank Folkmar Bornemann for valuable comments on a preliminary version of the present paper. We are grateful to the anonymous referees for their careful reading of our manuscript. Their suggestions to include the numerical experiments of~\cref{hist figure,CG tail fig} and add more details to the supersymmetric derivation significantly improved our paper.

\printbibliography%
\end{document}